\documentclass[12pt,a4paper]{article}

\usepackage{amsmath}
\usepackage{amsthm}
\usepackage{amsfonts}
\usepackage{amssymb}
\usepackage[english]{babel}

\newcommand{\di}{\displaystyle}

\newcommand{\ga}{\gamma}

\newcommand{\noa}{\noalign{\medskip}}
\newcommand{\la}{\lambda}

\newcommand{\qu}{\quad}

\newcommand{\al}{\alpha}
\newcommand{\be}{\beta}

\newcommand{\pa}{\partial}

\title{Flatness produced by some geometric PDEs }
\author{Iulia Hirica, Constantin Udriste, \\ Gabriel Pripoae, Ionel Tevy}

\date{}
\begin{document}
\maketitle

\newtheorem{definition}{Definition}[section]
\newtheorem{theorem}{Theorem}[section]
\newtheorem{proposition}{Proposition}[section]
\newtheorem{example}[theorem]{Example}
\newtheorem{corollary}[theorem]{Corollary}
\newtheorem{lemma}[theorem]{Lemma}
\newtheorem{remark}{Remark}[section]

\begin{abstract} This paper has several goals.
The first idea is to study the geometric PDEs of connection-flatness, curvature-flatness, Ricci-flatness,
scalar curvature-flatness in a modern and rigorous way. Although the idea is not new,
our main Theorems about flatness introduce a different point of view in Differential Geometry.
The second idea is to introduce and study the Euler-Lagrange
prolongations of PDEs-flatness solutions via associated least squares Lagrangian densities
and functionals on Riemannian manifolds.
All geometric PDEs turned into one of the most intensively
developing branches of modern differential geometry.
\end{abstract}



\section{PDEs in Differential Geometry}

The behavior of many different systems in nature and science are governed by a PDEs system.
Usually such a system is thought of in terms of coordinates in order to prove
existence of solutions or to find concrete ones.
However, tensorial PDEs in differential geometry contain also information which is independent
of the choice of coordinates. This is actually the
most important information as it is independent of any external structure, artificially added
to the PDE, and in this sense it is genuine. That is why, the differential geometry is
often considered as an "art of manipulating PDEs" \cite{E}, \cite{GV}, \cite{H}, \cite{HU},
\cite{SS}, \cite{S}, \cite{SY}.

To see how differential operators arose in differential geometry, we will discuss one
important problem here: given a manifold $M$, find a necessary and sufficient condition
for $M$ to admit an adapted metric or affine connection (eventually with certain curvature properties).

Specific aims in this paper are as follows:
(i) introducing those differential geometric structures needed to define and study geometric PDEs
(some of them in a manifestly coordinate-independent way);
(ii) define PDEs and their signification within Differential Geometry;
(iii) develop techniques to find intrinsic properties of PDEs;
(iv) discuss explicit examples to illustrate the importance of the choice of an
appropriate language.

While mathematicians were working on problems related to metric, connection,
Riemann tensor field, Ricci tensor field, scalar curvature,
it turned out that physicists, from other points of view, were also interested in similar
problems.

In the last time some journals sponsored special activities in differential geometry, with particular
emphasis on PDEs. As for example, Annals of Mathematics Studies published over 280 articles with impact
on the theory of overdetermined PDEs and harmonic analysis.

Our aim is to promote the use of the differential geometric language
to define, study and treat important geometric PDEs, and least squares Lagrangian
densities produced by some geometric differential operators, having in mind research collaborations
between pure and applied mathematicians interested in PDEs.

Let $M$ be an oriented manifold of dimension $n$. Any differential operator (tensorial or not)
on the Riemannian manifold $(M,g=(g_{ij}))$ and the metric (geometric structure)
$g$ generate a least squares Lagrangian density $L$.
The extremals of the Lagrangian $\mathcal{L}=L\sqrt{\det(g_{ij})}$,
described by Euler-Lagrange PDEs, include the solutions of
initial PDEs and other solutions which we call "Euler-Lagrange prolongations" of that solutions.

In this paper we develop a more geometric approach, explaining the
true mathematical meaning of some geometric PDEs and their Euler-Lagrange prolongations.
Although we have formulated what it means to have a geometric PDE, there are still
many avenues of research to pursue for the future.

{\bf Open problem}. Find symmetries or conservation laws of geometric flatness PDEs.

\section{Flatness produced by some geometric PDEs}

The most important geometric PDEs are those producing flatness: connection-flatness, curvature-flatness, Ricci-flatness,
scalar curvature-flatness. The first PDEs system is non-tensorial, while the last three PDEs systems are tensorial.

\subsection{Connection-flat manifolds}

The connection flatness and the curvature flatness are interconnected
\cite{C}, \cite{CV}, \cite{CVa}, \cite{TK}, \cite{TY}, \cite{V}.

The pair $(M,\nabla)$, where $\nabla$ is a symmetric connection,
with the components $\Gamma^i_{jk}$, $i,j,k = \overline{1,n}$, of class $C^\infty$,
is called affine manifold.
\begin{definition}
An affine manifold $(M,\nabla)$ is called connection-flat if around each point of $M$
there exist a chart such that $\Gamma^i_{jk} =0$.
\end{definition}
Let $(M, g)$ be a Riemannian manifold. The Riemannian metric $g$ of components $g_{ij}$
and its inverse $g^{-1}$ of components $g^{ij}$ determine (locally) the Christoffel symbols of the second kind
$$\Gamma^i_{jk}=\frac{1}{2}g^{il}\left(\frac{\partial g_{lj}}{\partial x^k}+ \frac{\partial g_{lk}}{\partial x^j} -\frac{\partial g_{jk}}{\partial x^l}\right) = \frac{1}{2}g^{il}(\delta^r_l \delta^s_j \delta^t_k + \delta^r_l \delta^s_k \delta^t_j - \delta^t_l \delta^r_j \delta^s_k)\frac{\partial g_{rs}}{\partial x^t}$$
(overdetermined elliptic partial differential operator).

For the next explanations we need to introduce the {\it space of Riemannian metrics}.
Denote $S^2T^*M$ as the vector bundle of all symmetric $(0,2)$-tensors on $M$
and let $S^2_+T^*M$ be the open subset of all the positive definite ones.
The set $S^2_+T^*M$ is used when we have PDEs with the unknown a Riemannian metric.

Let us consider the condition $\Gamma^i_{jk} = 0$ like a (non-tensorial) PDEs system
$$\frac{1}{2}g^{il}(\delta^r_l \delta^s_j \delta^t_k + \delta^r_l \delta^s_k \delta^t_j - \delta^t_l \delta^r_j \delta^s_k)\frac{\partial g_{rs}}{\partial x^t}=0\leqno(1)$$
on $S^2_+T^*M$, i.e., $\frac{n^2(n+1)}{2}$ distinct first order non-linear nonhomogeneous PDEs whose
unknowns are $\frac{n(n+1)}{2}$ functions $g_{ij}$ (positive definite tensor);
for $n>1$, overdetermined system of PDEs; for $n=1$, determined system. This PDEs system is symmetric in $j,k$.

\begin{theorem}
The only solutions of PDEs system (1) are the constant Riemannian metrics $g_{ij}(x)= c_{ij}$ (Euclidean manifold).
\end{theorem}
\begin{proof}
The previous non-linear PDEs are equivalent to the linear PDEs
$\frac{\partial g_{rs}}{\partial x^t}=0,$
since the matrices $g^{il}$ and  $T^{rst}_{ljk}=\delta^r_l \delta^s_j \delta^t_k + \delta^r_l \delta^s_k \delta^t_j - \delta^t_l \delta^r_j \delta^s_k$ are non-degenerate. Because $M$ is supposed to be connected, it follows the solutions $g_{ij}(x)= c_{ij}$.
\end{proof}

\subsection{Curvature-flat manifolds}
The curvature flatness was discussed in
\cite{C}, \cite{CV}, \cite{CVa}, \cite{TK}, \cite{TY}, \cite{V}
based on the idea of finding an adapted coordinate system.
We bring up another point of view, looking for suitable metrics and connections,
and not for adapted coordinate systems.

\subsubsection{Case of affine manifold}
Let $(M,\nabla)$ be an affine manifold, and $\Gamma^i_{jk}$, $i,j,k = \overline{1,n}$,
be the components of the symmetric connection $\nabla$.
The connection $\nabla$ determines the curvature tensor field $Riem^\nabla$ whose components are
$$R^l_{\,\,ijk}= \frac{\partial}{\partial x^j}\Gamma^l_{ik} - \frac{\partial}{\partial x^k}\Gamma^l_{ij}
+ \Gamma^l_{js} \Gamma^s_{ik} - \Gamma^l_{ks} \Gamma^s_{ij}$$
$$= \left(\delta^{p}_j \delta^s_k - \delta^{p}_k \delta^s_j\right)\left(\frac{\partial}{\partial x^p}\Gamma^l_{is}+ \Gamma^l_{pn} \Gamma^n_{is}\right)$$
$$= \left(\delta^{p}_j \delta^s_k - \delta^{p}_k \delta^s_j\right)\left(\frac{\partial}{\partial x^p}\Gamma^l_{is}- \Gamma^l_{sn} \Gamma^n_{ip}\right),$$
where $i,j,k,l = \overline{1,n}$. The operator
$P^{ps}_{jk}=\frac{1}{2}(\delta^{p}_j  \delta^s_k - \delta^{p}_k \delta^s_j)$ is a projection, i.e., $P^2=P$,
and is covariant constant.

The properties of curvature tensor are:
$$R^l_{\,\,ijk}+ R^l_{\,\,ikj}=0,\,\,R^l_{\,\,ijk}+ R^l_{\,\,jki}+ R^l_{\,\,kij}=0$$
$$R^l_{\,\,ijk,m}+R^l_{\,\,ikm,j}+R^l_{\,\,imj,k}=0 \,\,\,\hbox{(Bianchi)}$$
$$R^l_{\,\,ijk,m}+R^l_{\,\,kim,j}+R^l_{\,\,mkj,i}+R^l_{\,\,jmi,k}=0\,\,\,\hbox{(Veblen)}.$$

\begin{definition}
An affine manifold $(M,\nabla)$ is called curvature-flat if $Riem^\nabla =0$.
\end{definition}

It is wellknown the following \cite{C}, \cite{CV}, \cite{CVa}, \cite{V}
\begin{theorem}
The curvature tensor field $Riem^\nabla$ vanishes if and only if
the symmetric connection $\nabla$ is locally Euclidean, i.e., around each point of $M$
there exists a chart such that $\Gamma^i_{jk}=0$.
\end{theorem}
Intuitively, a curvature-flat manifold is one that "locally looks like Euclidean space".

The fact that the curvature tensor field of components $R^i_{\,\,jkl}$ is everywhere zero, can be interpreted
as a completely integrable (tensorial) PDEs system (see \cite{BU}, Bianchi conditions on curvature tensor)
$$(\delta^p_j \delta^s_k-\delta^p_k\delta^s_j)\frac{\partial}{\partial x^p}\Gamma^l_{is}= \Gamma^l_{ks} \Gamma^s_{ij} - \Gamma^l_{js} \Gamma^s_{ik},$$
having similar properties as curvature tensor field.
A constant connection of components $\Gamma^i_{jk}(x)=C^i_{jk}$ is solution if
$C^l_{ks} C^s_{ij} - C^l_{js} C^s_{ik}=0.$  Conversely, if $\Gamma^l_{ks} \Gamma^s_{ij} - \Gamma^l_{js} \Gamma^s_{ik}=0,$
then $\Gamma^l_{ik}(x)= \frac{\partial f^l_i}{\partial x^k}(x)$, with suitable matrix $f^l_i(x)$.

On the other hand, the same PDEs system can be written in two ways:
$$\left(\delta^{p}_j \delta^s_k - \delta^{p}_k \delta^s_j\right)\left(\frac{\partial}{\partial x^p}\Gamma^l_{is}+ \Gamma^l_{pn} \Gamma^n_{is}\right)=0$$
or
$$\left(\delta^{p}_j \delta^s_k - \delta^{p}_k \delta^s_j\right)\left(\frac{\partial}{\partial x^p}\Gamma^l_{is}- \Gamma^l_{sn} \Gamma^n_{ip}\right)=0.$$
Each system consists in $\frac{n^2(n^2-1)}{12}$ distinct  first order linear quadratic
PDEs whose unknowns are  $\frac{n^2(n+1)}{2}$ functions $\Gamma^i_{jk}$; for $n>7$,
overdetermined system; for $n<7$, undetermined system; for $n=7$, determined system.

There is no general algorithm for finding exact solutions of the previous system
but some methods used in other areas can be useful.

The trivial solution of the first previous PDEs system is $\Gamma^i_{jk}(x)=0$,
but there are also nonvanishing solutions as for example
$\Gamma^i_{jk}(x)=c^i_{jk}f(x)$, with convenient constant parameters $c^i_{jk}$ and
$\displaystyle f(x)= -\frac{1}{x^1+...+x^n +c}$. Also, given $c^i_{jk}$, we find
$\displaystyle c^l_{is}\frac{df}{f^2}+ c^l_{pn} c^n_{is}dx^p=0,$
and $f(x)$ is determined by the condition
$c^l_{is}= c^n_{is}(c^l_{1n} x^1+...+c^l_{nn} x^n)f(x).$

To the previous curvature flatness PDEs system, we attach two Riccati PDEs systems:
$$\frac{\partial}{\partial x^p}\Gamma^l_{is}(x)+ \Gamma^l_{pn}(x) \Gamma^n_{is}(x)=0,\leqno(2)$$
$$\frac{\partial}{\partial x^p}\Gamma^l_{is}(x)- \Gamma^l_{sn}(x) \Gamma^n_{ip}(x)=0.\leqno(3)$$
\begin{theorem}
The Riccati PDEs systems (2) and (3) are completely integrable.
\end{theorem}
\begin{proof}
The Riccati PDEs system (2) is completely integrable since the second order partial derivatives are equal.

The Riccati PDEs system (3) is completely integrable via Frobenius Theorem for the Pfaff system
$$d\Gamma^l_{is}(x)- \Gamma^l_{sn}(x) \Gamma^n_{ip}(x)dx^p=0$$
(integrant factor). Indeed, the $1$-forms
$$\omega^l_{is}=\delta^l_m\delta^p_i\delta^q_s\,\,\, d\Gamma^m_{pq}(x)- \Gamma^l_{sn}(x) \Gamma^n_{ip}(x)dx^p,$$
on the space $\mathbb{R}^m,\,\,\, m= \frac{n^2(n+1)}{2}+n$, of coordinates $(\Gamma^m_{pq}, x^i)$, are smooth
pointwise linearly independent forms (since the matrix of the components has the
rank $p=\frac{n^2(n+1)}{2}$) and there exists the smooth $1$-forms $a^{l\,\,\,\alpha \beta}_{is\,\,\,\gamma}$
such that
$$d\omega^l_{is}= a^{l\,\,\,\alpha \beta}_{is\,\,\,\gamma}\wedge \omega^\gamma_{\,\,\,\alpha\beta}.$$
Complete integrability means conditions for the global existence of connections, via Cauchy problems;
foliation by maximal integral manifolds.
\end{proof}

The previous Riccati PDEs are not tensor equations. That is, they are not invariant
with respect to the changing of the system of coordinates. Particularly, they admit constant
connection solutions $\Gamma^i_{jk}(x) = C^i_{jk}$, soliton solutions
$\Gamma^i_{jk}(x) = \phi^i_{jk}(a_lx^l)$ etc.

\begin{theorem}
Any solution $\Gamma^i_{jk}(x)$ of one of the Riccati PDEs systems (2) and (3)
is a solution of the curvature-flatness PDEs system.
\end{theorem}

\begin{corollary}
Suppose $\Gamma^l_{is}(x)$ is a solution of the PDEs system (2) and $T^l_{is}(x)$ is an arbitrary tensor field. The connection  $\Gamma^l_{is}(x)+T^l_{is}(x)$ is a solution of the PDEs system (2) if and only if
$$\frac{\partial}{\partial x^p}T^l_{is}(x)+ \Gamma^l_{pn}(x) T^n_{is}(x)+T^l_{pn}(x) \Gamma^n_{is}(x)+T^l_{pn}(x) T^n_{is}(x) =0$$
(Riccati PDEs).
\end{corollary}

\begin{theorem}
Suppose $\nabla$ is a symmetric connection of coefficients $\Gamma^i_{jk}$.
The following statements are equivalent.

(i) The affine manifold $(M,\nabla)$ is curvature-flat, i.e., $R^l_{\,\,ijk}(x)=0$.

(ii) For each point $x\in M$, there exists a chart such that $\Gamma^i_{jk}(x)=0$.

(iii) For each point $x\in M$, there exists a chart such that
$$\frac{\partial}{\partial x^p}\Gamma^l_{is}(x)=0,\,\,\Gamma^l_{pn}(x) \Gamma^n_{is}(x)=0.$$

(iv) For each point $x\in M$, there exists a chart such that
$$\frac{\partial}{\partial x^p}\Gamma^l_{is}(x)+ \Gamma^l_{pn}(x) \Gamma^n_{is}(x)=0.$$
\end{theorem}

\begin{proof} The diagram
$$\begin{array}{lcr} (i)&\rightarrow &(ii)\\
\uparrow &  & \downarrow\\
(iv)&\leftarrow &(iii)
\end{array}$$
is closed. Indeed $(i)\Rightarrow (ii)$ is the classical Theorem 2.2. Then the implications
$(ii)\Rightarrow (iii)\Rightarrow (iv)$ are automatically. The implication $(iv)\Rightarrow (i)$
is the Theorem 2.4.
\end{proof}

\begin{remark}
An equivalent Theorem can be formulated using the Riccati PDEs system (3). In this sense
the Theorem 2.5 can be replaced by a new Theorem adding the following two statements:

(v) for each point in $M$, there exists a chart with local coordinates
($x^1,...,x^n$) such that $  \nabla_{\partial_{x^i}}  \nabla_{\partial_{x^j}} \partial_{x^k} =0$,
for all indices $i,j,k = \overline{1,n}$;

(vi) for each point in $M$, there exists a chart with local coordinates
($x^1,...,x^n$) such that
$$\frac{\partial \Gamma^l_{is}}{\partial x^p} - \Gamma^l_{sn} \Gamma^n_{ip} = 0.$$

Obviously, property (v) is equivalent to property (iv).
The proof that (vi) is equivalent with (i)-(iv) is similar to the previous one.

The property (vi) is quite strange, as it is not derived (like (iv)) from a double covariant derivative (like (v)).
It deserves a deeper investigation on its own.

\begin{example}
Let $M$ be an $n$-dimensional parallelizable manifold
and let $ E_1, ..., E_n$ be  a basis of vector fields on $M$. One knows there
exists a unique affine connection $\nabla^{(-)}$ on $M$ such that $ \nabla^{(-)}_{E_i} E_j = 0$
for any $i,j= \overline{1,n}$. This Cartan-Schouten connection is flat but, in general, it is not symmetric.

We look for the affine connections satisfying
$$ \nabla_{E_i}\nabla_{E_j} E_k = 0,\leqno (4)$$
for any $i,j,k= \overline{1,n}$. Obviously, $\nabla^{(-)}$ is such a connection.

We derive that $\nabla$ satisfies (4) iff its components w.r.t the given basis verify
$$ E_i(\Gamma^l_{jk})+ \Gamma^s_{jk} \Gamma^l_{is} =0, \leqno (5)$$
for any $i,j,k,l= \overline{1,n}$.

The relation (5) extends the property (iv) from the Theorem 2.5.
\end{example}

\begin{example} In particular, let $G$ be an $n$-dimensional Lie group and
$E_1$, ..., $E_n$ be  a basis of left invariant vector fields in the Lie algebra
$L(G)$. Suppose moreover that the connections are left invariant. Then the condition (5) writes
$$ \Gamma^s_{jk} \Gamma^l_{is} =0,\leqno(6)$$ for any $i,j,k,l= \overline{1,n}$.

This shows that the set of all such (left invariant) connections is a "universal"
generalized cone in $\mathbb{R}^{n^3}$, whose vertex is $\nabla^{(-)}$.

In this particular case, the relation (6) corresponds to the properties (iii)
and (iv) from the Theorem 2.5 (it is not more general nor particular, it is similar!).
\end{example}
\end{remark}

{\bf Open problem}. What happens if the connection $\nabla$ is not symmetric?
\medskip

The general solution for a PDE is not useful for solving particular problems with physical content.
However, it may be pointed out that some solutions with a certain generality have played very
important roles in the discussion of physical problems (for example solutions of Einstein PDEs,
soliton solutions etc).

\subsubsection{Case of Riemannian manifold}
Let $(M, g = (g_{ij}))$ be a Riemannian manifold. The Riemannian metric $(g_{ij})$
determines the Riemannian curvature tensor field $Riem^g$ of components
$$R_{ijkl} =\frac{1}{2}\left(\frac{\partial^2 g_{ik}}{\partial x^j\partial x^l}+\frac{\partial^2 g_{jl}}{\partial x^i\partial x^k}  - \frac{\partial^2 g_{jk}}{\partial x^i\partial x^l}-\frac{\partial^2 g_{il}}{\partial x^j\partial x^k}\right)
+g_{mn}(\Gamma^m_{jk}\Gamma^n_{il}- \Gamma^m_{jl}\Gamma^n_{ik})$$
$$=\frac{1}{2} \delta^p_{[i} \delta^q_{j]} \delta^r_{[k} \delta^s_{l]} \,\,\frac{\partial^2 g_{pr}}{\partial x^q \partial x^s}+g_{mn}\delta^q_j\delta^p_i \delta^r_{[k} \delta^s_{l]}\Gamma^m_{qr} \Gamma^n_{ps},$$
where
$$\delta^p_{[i} \delta^q_{j]}=\delta^p_{i} \delta^q_{j}-\delta^p_{j} \delta^q_{i},\,\, \Gamma^i_{jk}=\frac{1}{2}g^{il}(\delta^r_l \delta^s_j \delta^t_k + \delta^r_l \delta^s_k \delta^t_j - \delta^t_l \delta^r_j \delta^s_k)\frac{\partial g_{rs}}{\partial x^t}.$$

Properties:
$$R_{ijkl}= R_{klij},\,\,\,R_{ijkl}=-R_{ijlk}=-R_{jikl}$$
$$R_{ijkl}+ R_{iljk}+R_{iklj}=0\,\,\,\hbox{(first Bianchi identity)}$$
We do not mention here the 2nd Bianchi identity
since it involves covariant derivatives and hence it is an equation
not only on the Riemannian curvature tensor field but also on the original metric.

In this case Riemannian curvature flatness condition means the tensorial PDEs system
$$\frac{1}{2} \delta^p_{[i} \delta^q_{j]} \delta^r_{[k} \delta^s_{l]} \,\,\frac{\partial^2 g_{pr}}{\partial x^q \partial x^s}+g_{mn}\delta^q_j\delta^p_i \delta^r_{[k} \delta^s_{l]}\Gamma^m_{qr} \Gamma^n_{ps}=0,$$
on $S^2_+T^*M$, with $\frac{n^2(n^2-1)}{12}$ distinct second order linear non-homogeneous PDEs whose unknowns
are $\frac{n(n+1)}{2}$ functions $g_{ij}$ (positive definite tensor);
for $n<3$, undetermined system; for $n>3$, overdetermined system; for $n=3$, determined system. This PDEs system is
parabolic since the set of eigenvalues of the matrix $T^{pqrs}_{ijkl}=\delta^p_{[i} \delta^q_{j]} \delta^r_{[k} \delta^s_{l]}$
(tensorial product of a matrix by itself) contains the eigenvalue $0$.
Indeed all eigenvectors, respectively eigenvalues are: $X^{ijkl}$-symmetric in $(i,j)$ or in $(k,l)$,
with $\lambda=0$; $X^{ijkl}$-skewsymmetric in $(i,j)$ and in $(k,l)$ with $\lambda=2$. Of course,
this PDEs system has all properties of curvature tensor field.

Any solution $g_{ij}$ (positive definite tensor) of the strong PDEs system
$$\frac{1}{2} \delta^p_{[i} \delta^q_{j]} \,\,\frac{\partial^2 g_{pr}}{\partial x^q \partial x^s}+g_{mn} \Gamma^m_{jr} \Gamma^n_{is}=0,$$
on $S^2_+T^*M$, is a solution of the curvature flatness PDE system. This PDEs system is parabolic since
the set of eigenvalues of the projection
$P^{pq}_{ij}= \frac{1}{2} \delta^p_{[i} \delta^q_{j]}$
contains $0$. Indeed all the eigenvectors and eigenvalues  are: $X^{ij}$-symmetric with $\lambda=0$
and $X^{ij}$-skewsymmetric with $\lambda=1$.

The trivial solution of the previous PDEs system is $g_{ij}(x) = c_{ij}$,
but there are also nonconstant solutions as for example
$g_{ij}(x) = c_{ij}f(x)$, with convenient function $f(x)$,
soliton solutions $g_{ij}(x) = \phi_{ij}(a_kx^k)$ etc.

\subsection{Solution via normal coordinates}

Let $(M,\nabla)$ be  a differentiable manifold equipped with a symmetric affine connection.
Normal coordinates at a point $p$ are a local coordinate system in a neighborhood of $p$ obtained by
applying the exponential map to the tangent space at $p$. In a
normal coordinate system, the components $\Gamma^i_{jk}$ of the connection
vanish at the point $p$, thus often simplifying local judgments.

In normal coordinates associated to the Levi-Civita connection of a
Riemannian manifold $(M,g)$, one can additionally arrange that the
metric tensor is the Kronecker delta at the point $p$, and
that the first partial derivatives of the metric at $p$ vanish.

Giving $R_{ijkl}$, we look to solve the PDEs
$$\frac{1}{2} \delta^p_{[i} \delta^q_{j]} \delta^r_{[k} \delta^s_{l]} \,\,\frac{\partial^2 g_{pr}}{\partial x^q \partial x^s}+g_{mn}\delta^q_j\delta^p_i \delta^r_{[k} \delta^s_{l]}\Gamma^m_{qr} \Gamma^n_{ps}=R_{ijkl},$$
trying for series solution
$$g_{jl}(x)= g_{jl}(0)+ \frac{\partial g_{jl}}{\partial x^k}(0)x^k +\frac{1}{2}\frac{\partial^2 g_{jl}}{\partial x^i \partial x^k}(0)x^ix^k +...$$

In the center of normal coordinates, the components of curvature tensor field $Riem^g$ are
$$R_{ijkl}(0) =\frac{1}{2}\left(\frac{\partial^2 g_{ik}}{\partial x^j\partial x^l}+\frac{\partial^2 g_{jl}}{\partial x^i\partial x^k}  - \frac{\partial^2 g_{jk}}{\partial x^i\partial x^l}-\frac{\partial^2 g_{il}}{\partial x^j\partial x^k}\right)(0).$$
Now we observe that the Cauchy problem
$$\partial^2_{ik}g_{jl}=\frac{1}{3}\left(R_{ijkl} + R_{ilkj}\right),\,\,\,g_{ij}(0)=\delta_{ij},\,\,\,\partial_kg_{ij}(0)=0$$
determine a solution $g_{jl}(x)$.

To summarize, given any tensor $R_{ijkl}$ satisfying the symmetries stated above, any smooth metric of the form
$$g_{jl}(x)= g_{jl}(0)+\frac{1}{3}\left(R_{ijkl} + R_{ilkj}\right)x^ix^k + O(|x|^3)$$
has curvature tensor $R_{ijkl}$ at $x=0$. This solution was commented
in mathematical discussions on MathOverflow, 2019
(Riemann's formula for the metric in a normal neighborhood ...).

\begin{remark} Suppose $(M,g)$ is a Riemannian manifold with $n=dim M \geq 4$.
If the curvature tensor vanishes identically, then the Riemannian manifold
is conformally flat. This statement is wellknown and is based on the Weyl curvature tensor field
$$C_{iklm}= R_{iklm} + \frac{1}{n-2}(R_{im}g_{kl}-R_{il}g_{km}+R_{kl}g_{im}-R_{km}g_{il})$$
$$+\frac{R}{(n-1)(n-2)}(g_{il}g_{km}-g_{im}g_{kl}).$$
For dimensions $2$ and $3$, the Weyl curvature tensor vanishes identically.
 For dimensions $\geq 4$, the Weyl curvature is generally nonzero.
 If the Weyl tensor vanishes in dimension $\geq 4$, then the metric is
 locally conformally flat: there exists a local coordinate system in which
 the metric tensor is proportional to a constant tensor.
\end{remark}

The question of whether a geometric manifold is flat depends on its structure (metric, affine connection)
and only indirectly on its topology. A topological space equipped with one metric may be flat,
but equipped with another it may not be.

\subsection{Ricci-flat manifolds}
Ricci flatness was described in
\cite{DHK}, \cite{FW}, \cite{SKM}, \cite{G}, \cite{W}, \cite{ZK}, \cite{HeHe}
underlining locally the difference between an "Euclidean ball" and a "geodesic ball".

In Physics, Ricci-flat manifolds represent vacuum solutions to the analogues
of Einstein's equations for Riemannian manifolds of any dimension,
with vanishing cosmological constant.

\subsubsection{Case of equiaffine manifold}
A torsion-free affine connection $\nabla$, of components $\Gamma^i_{jk}$, is called (locally) equiaffine if
locally there is a volume form (nonvanishing $n$-form) $\omega$ that is parallel with respect to $\nabla$.
An affine connection $\nabla$ with zero torsion is Ricci-symmetric if and only if $\nabla$ is locally equiaffine.

Let $(M,\nabla)$ be an equiaffine manifold. The components $R_{ik}$ of the Ricci tensor field $Ric^\nabla$
are obtained by the contraction of the first and third indices of the curvature tensor field $R^l_{\,\,ijk}$, i.e.,
$$R_{ik} = R^l_{\,\,ilk}= \frac{\partial}{\partial x^l}\Gamma^l_{ik} - \frac{\partial}{\partial x^k}\Gamma^l_{il}
+ \Gamma^l_{ls} \Gamma^s_{ik} - \Gamma^l_{ks} \Gamma^s_{il}$$
$$=\left(\delta^{p}_q \delta^s_k - \delta^{p}_k \delta_q^s  \right)\left(\frac{\partial}{\partial x^p}\Gamma^q_{is}+
\Gamma^q_{pn} \Gamma^n_{is}\right)$$
$$= \left(\delta^{p}_q \delta^s_k - \delta^{p}_k \delta^s_q\right)\left(\frac{\partial}{\partial x^p}\Gamma^q_{is}- \Gamma^q_{sn} \Gamma^n_{ip}\right).$$

\begin{definition}
An equiaffine manifold $(M,\nabla)$ is called Ricci-flat if  $Ric^\nabla =0$.
\end{definition}

On the other hand, the condition that the Ricci curvature vanishes, either
$$\left(\delta^{p}_q \delta^s_k - \delta^{p}_k \delta_q^s  \right)\left(\frac{\partial}{\partial x^p}\Gamma^q_{is}+
\Gamma^q_{pn} \Gamma^n_{is}\right)=0$$
or
$$\left(\delta^{p}_q \delta^s_k - \delta^{p}_k \delta^s_q\right)\left(\frac{\partial}{\partial x^p}\Gamma^q_{is}- \Gamma^q_{sn} \Gamma^n_{ip}\right)=0$$
is a system of $\frac{n(n+1)}{2}$ distinct first order divergence
quadratic tensorial PDEs with $\frac{n^2(n+1)}{2}$ unknown functions $\Gamma^i_{jk}$;
for $n>1$, undetermined system; for $n=1$, determined system.
Here $\mathcal{P}^{ps}_{qk}=\delta^{p}_q \delta^s_k - \delta^{p}_k \delta_q^s $ works like a trace between $p$ and $q$,
in order to produce a divergence term. This operator is  associated to the projection $P$.
Any divergence PDE represents a conservation law.
\begin{theorem}
Any solution $\Gamma^i_{jk}$ of the completely integrable Riccati PDE systems (2) or (3)
is a solution of the Ricci-flatness PDE system.
\end{theorem}

\begin{theorem} Let $\nabla$ be an arbitrary equiaffine connection with components $\Gamma^i_{jk}$.
Let $\Omega_+$ be the set of partial derivative operators
(exotic objects) of the form $\frac{\partial}{\partial x^p}\Gamma^q_{is}+
\Gamma^q_{pn} \Gamma^n_{is}$, and $\Omega_-$ be the set of partial derivative operators
(exotic objects) of the form $\frac{\partial}{\partial x^p}\Gamma^q_{is}- \Gamma^q_{sn} \Gamma^n_{ip}$.
Let $P$ be the previous projection and $\mathcal{P}$ the associated trace.
Let $\hbox{Tr}$  be the trace with respect to indices $i$ and $k$. Then $\mathcal{P}= Tr\circ P$.
\end{theorem}

\begin{proof} The diagram
$$\begin{array}{lcr} \Omega_\pm&\stackrel{P}\rightarrow &R^i_{\,\,jkl}\\
\mathcal{P}\searrow  &  & \swarrow Tr\\
&R_{jl} &
\end{array}$$
clarifies the statement in the Theorem.
\end{proof}

{\bf Open problem}. What happens if the connection $\nabla$ is not equiaffine?

\subsubsection{Case of Riemannian manifold}
In case that $(M,g=(g_{ij}))$ is a Riemannian manifold, the Ricci tensor field $Ric^g$ has the components
$$R_{ik} = \frac{\partial \Gamma^l_{ik}}{\partial x^l}-\Gamma^m_{il} \Gamma^l_{km}-\nabla_k\left(\frac{\partial}{\partial x^i}   \left(\ln \sqrt{\det(g_{mn})}\right)\right),$$
where
$$\Gamma^i_{jk}=\frac{1}{2}g^{il}(\delta^r_l \delta^s_j \delta^t_k + \delta^r_l \delta^s_k \delta^t_j - \delta^t_l \delta^r_j \delta^s_k)\frac{\partial g_{rs}}{\partial x^t}.$$
Now Ricci-flat manifold
$$\frac{\partial \Gamma^l_{ik}}{\partial x^l}-\Gamma^m_{il} \Gamma^l_{km}-\nabla_k\left(\frac{\partial}{\partial x^i}   \left(\ln \sqrt{\det(g_{mn})}\right)\right)=0$$
means $\frac{n(n+1)}{2}$ distinct PDEs with $\frac{n(n+1)}{2}$ unknown functions $g_{ij}$, on $S^2_+T^*M$.
They are special cases of Einstein manifolds, where the cosmological constant vanish.

Since Ricci curvature in a Riemannian manifold measures the amount by which the volume of a
small geodesic ball deviates from the volume of a ball in Euclidean space.
This statement is made precise via the formula
$$\omega_{1,2,...,n}(x)=\left(1-\frac{1}{6}R_{ij}(m)x^ix^j + O(\|x\|^3)\right)\omega_{1,2,...,n}(m),$$
(see Gray \cite{G}), where $\omega$ is the metric volume $n$-form for the metric tensor $g$ defined on $M$ and
Riemann normal coordinates are used.

\begin{remark}
Let $(M,g_{ij})$ be a Riemannian manifold.
If the dimension of $M$ is $3$, then Ricci flatness implies Riemann flatness.

This statement is wellknown and it is based on the fact that
for $3$-dimensional Riemannian manifolds $(M,g_{ij})$, we have
$$R_{ijkl}= R_{ik}g_{jl}-R_{il} g_{jk} + g_{ik} R_{jl}- g_{il}R_{jk}-\frac{R}{2}(g_{ik} g_{jl}- g_{il} g_{jk}).$$
\end{remark}

\subsubsection{Solution in harmonic coordinates}

Harmonic coordinates in higher dimensions were developed initially in the context of
general relativity by Einstein, and then by Lanczos to simplify the formula for the Ricci tensor. They were then
introduced into Riemannian geometry by Sabitov and Sefel \cite{SS} and later were studied
by DeTurck and Kazdan \cite{TK} and  Hebey and Herzlich \cite{HeHe}. The essential
motivation for introducing a harmonic coordinate
system is that the Ricci tensor is simplified, when written in this coordinate system.

\begin{definition}
A coordinate chart $(x^1, ..., x^n)$ on a Riemannian manifold $(M,g=(g_{ij}))$
is called harmonic if $\Delta x^k = 0$ for all $k = \overline{1,n}$. Since
$\Delta x^k= - g^{ij}(x)\Gamma^k_{ij}(x)$, where the $\Gamma'$s are the Christoffel symbols of the connection
associated to $g$, we get that a coordinate chart $(x^1, ..., x^n)$ is harmonic if
and only if $g^{ij}(x) \Gamma^k_{ij}(x) = 0$, for any $i,j,k = \overline{1,n}$.
\end{definition}

In a harmonic coordinate system, the Ricci tensor has a simplified formula
(quasilinear elliptic operator)
$$R_{ij}=g^{kl}\left(-\frac{1}{2}\,\,\partial_k\partial_l g_{ij}+ g_{mn}\Gamma^m_{ik}\Gamma^n_{lj}\right).$$

Given $R_{ij}=0$, we look to analyse and to solve the PDEs system
$$g^{kl}\left(-\frac{1}{2}\,\,\partial_k\partial_l g_{ij}+ g_{mn}\Gamma^m_{ik}\Gamma^n_{lj}\right)=0,$$
 on $S^2_+T^*M$, in harmonic coordinates. The strong form associated to this PDEs system is
$$-\frac{1}{2}\,\,\partial_k\partial_l g_{ij}+ g_{mn}\Gamma^m_{ik}\Gamma^n_{lj}=0.$$
Any Cauchy problem attached to this strong second order PDEs system has unique solution.
Some examples: constant solution $g_{ij}(x) = c_{ij}$;
nonconstant solutions, $g_{ij}(x) = c_{ij}f(x)$, with convenient function $f(x)$;
soliton solutions $g_{ij}(x) = \phi_{ij}(a_kx^k)$; etc.

\subsection{Scalar curvature - flat manifold}
Let $\nabla$ be an equiaffine connection of components $\Gamma^i_{jk}$ and $g$ be
a Riemannian metric of components $g_{ij}$.
\begin{definition}
An equiaffine Riemannian manifold $(M,\nabla,g)$ is called scalar curvature-flat if
around each point of $M$ there exists a chart
such that $\mathcal{R}=g^{ij}R_{ij} = 0$.
\end{definition}

On an equiaffine Riemannian manifold $(M,\nabla,g)$, we can look at the condition $\mathcal{R}=0$
as one PDE with $\frac{n^2(n+1)}{2}$ unknown functions $\Gamma^i_{jk}$. On Riemannian manifold
$(M,g)$, the condition $R=0$ is one PDE with $\frac{n(n+1)}{2}$ unknowns functions $g_{ij}$, on $S^2_+T^*M$.

In dimension two, Riemannian curvature tensor, Ricci curvature and
scalar curvature are all the same. By rescaling a metric at each point,
namely a conformal change, we can make its curvature constant. This is
the uniformization theorem which identifies conformal structures
(or equivalently complex structures) and Einstein metrics in this dimension.
In dimension three, every Einstein metric has constant sectional curvature.
The search for them is an important avenue into understanding
the Poincare conjecture and Thurston geometrization conjecture for three
dimensional topology and we have the recent breakthrough by Perelman
toward solving them.

\section{Least squares Lagrangian densities}
The examples presented in this Section include typical functionals
that appear in the theory of geometric and physical fields.
Similar problems are found in \cite{An} and \cite{BU}, where
the approach that has been used is the so-called variational approach.

Generally, the Euler–Lagrange equation provides the
equation of motion for the dynamical field specified in the
Lagrangian.

{\bf Dual variational principle}: Let $(M,g)$ be a Riemannian manifold. Usually, the local components
of the metric $g$ are denoted by $g_{ij}$ and the components of the inverse $g^{-1}$ are denoted by $g^{ij}$.
Due to the musical isomorphism between the tangent bundle $TM$ and the cotangent bundle $T^*M$
of a Riemannian manifold induced by its metric tensor $g$, the arbitrary variations of $g_{ij}$ are equivalent
to the arbitrary variations of $g^{ij}$, and any Lagrangian with respect to $g_{ij}$
can be regarded as a Lagrangian in relation to $g^{ij}$, but the orders are different.

When calculating the variation with respect to $g^{ij}$, certain terms may appear
whose integral over any domain $\Omega$ can
be reduced via integration by parts to an integral over the boundary $\partial \Omega$, which vanish
(variations vanish on boundary).
Modulo this statement, the Euler-Lagrange PDEs are reduced to $\frac{\partial \mathcal{L}}{\partial g^{ij}}=0$
(the formal partial derivatives equal to zero).

\subsection{Riemannian volume form}
Suppose $(M,g=(g_{ij}))$ is a smooth oriented Riemannian manifold.
Then there is a consistent way to choose the sign of the
square root $\sqrt{\det (g_{ij})}$ and define a volume form
$d\mu= \sqrt{\det (g_{ij})}\,\,dx^1\wedge\cdots\wedge dx^n.$
We call it the Riemannian volume form of $(M,g)$. Having a volume form
allows us to integrate functions on $M$. In particular $vol(M) = \int_M d\mu$ is an
important invariant of $(M,g)$. It also allows us to define an inner product
$\langle \phi, \psi \rangle= \int_M \langle\phi(x),\psi(x)\rangle_g\,\,d\mu,$
on the space of differential forms and other tensors or objects on $M$, using the
metric $g$ and its inverse $g^{-1}$. This inner product induces the square of the norm
$\|\phi\|^2=\int_M |\phi(x)|^2_g\,\,d\mu.$

\subsection{Positive definite differential operator}
For an $n\times n$ matrix of numbers or functions, positive definiteness is equivalent to the fact that
its leading principal minors are all positive ($n$ inequalities).

For an $n\times n$ matrix of partial derivatives operators, positive definiteness is equivalent to the fact that
its leading principal minors are all positive ($n$ partial differential inequalities).
See also \cite{JS}.

\subsection{$\Gamma$-flatness deviation}

Consider the triple $(M,\nabla,g)$, where $\nabla$ is a symmetric affine connection of components
$\Gamma^i_{jk}$ and  $(M,g=(g_{ij}))$ is a smooth oriented Riemannian manifold.
The square of the norm $L=\|\nabla\|^2= g_{ip}g^{jq}g^{kr}\Gamma^i_{jk}\Gamma^p_{qr}$
(Lagrangian density of zero order with respect to $\Gamma^i_{jk}$) determines the functional ($\nabla$-flatness deviation)
$ I(\nabla)=\int_M \|\nabla\|^2 d\mu,$
whose associated Lagrangian is $\mathcal{L}= \|\nabla\|^2\sqrt{\det(g_{ij})}$.

\begin{theorem}
The solutions of Euler-Lagrange PDEs of the Lagrangian $\mathcal{L}$
are only global minimum points $\Gamma^p_{qr}= 0$.
\end{theorem}
\begin{proof}
We use the partial derivatives
$\frac{\partial \Gamma^i_{jk}}{\partial \Gamma^s_{mn}}= \delta^i_s \delta^m_j \delta^n_k.$
The Euler-Lagrange equations
$\frac{\partial \mathcal{L}}{\partial \Gamma^s_{mn}}=2g_{sp}g^{mq}g^{nr}\Gamma^p_{qr}\sqrt{\det(g_{ij})}=0$
are equivalent to $\Gamma^p_{qr}= 0$.
\end{proof}

Let $(M,g=(g_{ij}))$ be a Riemannian manifold and $\Gamma^i_{jk}$ the induced Christoffel symbols
of the Riemannian connection $\nabla$.
The square of the norm $L=\|\nabla\|^2= g_{ip}g^{jq}g^{kr}\Gamma^i_{jk}\Gamma^p_{qr}$
is a Lagrangian density of first order with respect to $g_{ij}$ and of order zero with respect to $g^{ij}$.
That is why, we have two kinds of functionals describing $\nabla$-flatness deviation,
either $ I(g)=\int_M \|\nabla\|^2 d\mu$ or $ I(g^{-1})=\int_M \|\nabla\|^2 d\mu.$
Though the second is more simple, from variational point of view, let us begin the study with the first functional
whose associated Lagrangian $\mathcal{L}= \|\nabla\|^2\sqrt{\det(g_{ij})}$ is of first order in $g_{ij}$.
\begin{theorem}
The extremals $g$ of $I(g)$, i.e., the solutions of PDEs
$$\left[g^{jq}g^{kr}\Gamma^m_{jk}\Gamma^n_{qr}- g_{ip}\Gamma^i_{jk}\Gamma^p_{qr}(g^{mj}g^{nq}g^{kr}+g^{jq}g^{mk}g^{nr})\right]\sqrt{\det(g_{ij})}$$
$$ +\frac{1}{2}g_{ip}g^{jq}g^{kr}\Gamma^i_{jk}\Gamma^p_{qr}g^{mn}\sqrt{\det(g_{ij})}$$
$$- D_{x^l}\left[g^{jq}g^{kr}(\delta^m_u\delta^n_j \delta^l_k + \delta^m_u \delta^n_k\delta^l_j-\delta^l_u\delta^m_j\delta^n_k)\Gamma^u_{qr}\sqrt{\det(g_{ij})}\right]=0$$
are minimum points of $I(g)$.
\end{theorem}

\begin{proof} The extremals of the Lagrangian $\mathcal{L}$ are solutions of Euler-Lagrange PDEs
$$\frac{\partial \mathcal{L}}{\partial g_{mn}}- D_{x^l} \,\,\frac{\partial \mathcal{L}}{\partial(\partial_{x^l}g_{mn})}=0.$$
This critical points are global (when $\mathcal{L}$=0) or local (when $\mathcal{L}\neq 0$).

Suppose $\mathcal{L}\neq 0$. Based on obvious formulas
$$\frac{\partial g_{jk}}{\partial g_{mn}}= \delta^m_j \delta^n_k,
\,\,\frac{\partial g^{jk}}{\partial g_{mn}}= - g^{mj}g^{nk},$$
$$\frac{\partial}{\partial g_{mn}}\det(g_{ij})= \det(g_{ij})g^{mn},
\,\,\,\,\frac{\partial{(\partial_{x^t} g_{rs})}}{\partial{(\partial_{x^l}g_{mn})}}=\delta^l_t \delta^m_r \delta^n_s,$$
we obtain
$$\frac{\partial \mathcal{L}}{\partial g_{mn}}= \left[g^{jq}g^{kr}\Gamma^m_{jk}\Gamma^n_{qr}- g_{ip}\Gamma^i_{jk}\Gamma^p_{qr}(g^{mj}g^{nq}g^{kr}+g^{jq}g^{mk}g^{nr})\right]\sqrt{\det(g_{ij})}$$
$$ +\frac{1}{2}g_{ip}g^{jq}g^{kr}\Gamma^i_{jk}\Gamma^p_{qr}\sqrt{\det(g_{ij})}\,\,g^{mn},$$
$$\frac{\partial \mathcal{L}}{\partial{(\partial_{x^l}g_{mn})}}
= 2g_{ip}g^{jq}g^{kr}\Gamma^p_{qr}\sqrt{\det(g_{ij})}\,\,\frac{\partial \Gamma^i_{jk}}{\partial(\partial_{x^l}g_{mn})}$$
$$= g_{ip}g^{jq}g^{kr}\Gamma^p_{qr}g^{iu}(\delta^r_u \delta^s_j \delta^t_k + \delta^r_u \delta^s_k \delta^t_j - \delta^t_u \delta^r_j \delta^s_k)  \sqrt{\det(g_{ij})}\,\, \frac{\partial{(\partial_{x^t} g_{rs})}}{\partial{(\partial_{x^l}g_{mn})}}$$
$$
=g^{jq}g^{kr}\Gamma^u_{qr} (\delta^r_u \delta^s_j \delta^t_k + \delta^r_u \delta^s_k \delta^t_j - \delta^t_u \delta^r_j \delta^s_k)\delta^l_t \delta^m_r \delta^n_s \sqrt{\det(g_{ij})}
$$
$$= g^{jq}g^{kr}(\delta^m_u\delta^n_j \delta^l_k + \delta^m_u \delta^n_k\delta^l_j-\delta^l_u\delta^m_j\delta^n_k)\Gamma^u_{qr}\sqrt{\det(g_{ij})}.$$
The explicit Euler-Lagrange PDEs are those in Theorem.

Now let us compute the Hessian matrix of components
$$H_{(lmn)(abc)}= \frac{\partial^2 L}{\partial{(\partial_{x^l}g_{mn})}\partial{(\partial_{x^a}g_{bc})}}=g^{jq}g^{kr}(\delta^m_u\delta^n_j \delta^l_k + \delta^m_u \delta^n_k\delta^l_j-\delta^l_u\delta^m_j\delta^n_k)
\frac{\partial \Gamma^u_{qr}}{\partial(\partial_{x^a}g_{bc})}$$
$$=\frac{1}{2}g^{jq}g^{kr}(\delta^m_u\delta^n_j \delta^l_k + \delta^m_u \delta^n_k\delta^l_j-\delta^l_u\delta^m_j\delta^n_k)
g^{uv}(\delta^r_v \delta^s_q \delta^t_r + \delta^r_v \delta^s_r \delta^t_q - \delta^t_v \delta^r_q \delta^s_r)\frac{\partial{(\partial_{x^t}g_{rs})}}{\partial{(\partial_{x^a}g_{bc})}}$$
$$=\frac{1}{2}g^{jq}g^{kr}g^{uv}(\delta^m_u\delta^n_j \delta^l_k + \delta^m_u \delta^n_k\delta^l_j-\delta^l_u\delta^m_j\delta^n_k)
(\delta^b_v \delta^c_q \delta^a_r + \delta^b_v \delta^c_r \delta^a_q - \delta^a_v \delta^b_q \delta^c_r).$$
This matrix is invariant if one interchange $l$ with $a$ and the (un-ordered) pair $m, n$ with the (un-ordered)
pair $b, c$, what must happen with a mixed derivative. Since the matrix $H$ is positive definite,
all extremals are minimum points (Legendre-Jacobi criterium).
\end{proof}
The extremals of $I(g)$ are Euler-Lagrange prolongations of the solutions $g_{ij}(x)= c_{ij}$ (constants).

For comparison we use
${\cal L}(g^{-1})=g_{ip}g^{jq}g^{kr}\Gamma ^i_{jk}\Gamma ^p_{qr}\sqrt{det (g_{ij})},$
the general form of Euler-Lagrange PDEs $\frac{\partial {\cal L}}{\partial g^{mn}}=0$
and we formulate the following
\begin{theorem}
The extremals $g=(g_{ij})$ of ${\cal L}(g^{-1})$ are solutions of PDEs system
$$
g^{kr}\left(-2g_{mp}g_{ni}g^{jq}+g_{mi}g_{np}g^{jq}- 2g_{ip}\delta^j_m\delta^q_n +
\frac{1}{2}g_{ip}g^{jq}g_{mn}\right)\Gamma ^i_{jk}\Gamma ^p_{qr}=0.$$
\end{theorem}

For calculus of the matrix
$H_{(ab);(mn)}=\frac{\partial^2 \mathcal{L}}{\partial g^{ab}\partial g^{mn}}$,
we need $\frac{\partial g^{kr}}{\partial g^{ab}}= \delta^k_a\delta^r_b$ and
$\frac{\partial g_{mj}}{\partial g^{ab}}= -g_{ma}g_{bj}$.
We find
$$H_{(ab);(mn)}= 2g_{ip}\Gamma ^i_{ma}\Gamma ^p_{nb}$$
$$+g^{kr}
\Big[2g_{ap}g_{bi}\Gamma^i_{mk}\Gamma ^p_{nr}+
(4g_{mp}g_{ni}-g_{ip}g_{mn}-2g_{mi}g_{np})\Gamma ^i_{ak}\Gamma ^p_{br}
$$
$$+\frac{1}{2}
g^{jq}(g_{ip}g_{ma}g_{nb}-g_{ap}g_{mn}g_{ib})\Gamma ^i_{jk}\Gamma ^p_{qr}\Big].$$
This matrix is not (neither positive nor negative) definite since it vanishes in the center of normal coordinates.
This is why this matrix is of no help in determining that extremals could be extremes.

\subsection{Riemann-flatness deviation}

Let $\nabla$ be a symmetric connection of components $\Gamma^i_{jk}$ and $g$
be a Riemannian metric of components $g_{ij}$. On the smooth oriented manifold $(M,\nabla,g)$, we introduce
the square of the norm $L=\|Riem^\nabla\|^2$ $= g_{ip}g^{jq}g^{kr}g^{ls}R^i_{\,\,jkl}R^p_{\,\,qrs}$,
$R^i_{\,\,jkl}= 2P^{ps}_{kl} \left(\frac{\partial}{\partial x^p}\Gamma^i_{js}+ \Gamma^i_{pn} \Gamma^n_{js}\right)$
which is a Lagrangian density of first order in $\Gamma^i_{jk}$.
It determines a functional (Riemann - flatness deviation) similar to the Yang-Mills functional, namely
$I(\nabla)=\int_M \|Riem^\nabla\|^2d\mu.$
The extremals of $\mathcal{L}(\nabla , \partial \nabla )= \|Riem^\nabla\|^2 \sqrt{\det(g_{ij})}$ are solutions of Euler-Lagrange PDEs
$\frac{\partial \mathcal{L} }{\partial \Gamma^u_{vw}}- D_{x^t}\frac{\partial \mathcal{L} }{\partial ({\partial_{x^t}} \Gamma^u_{vw})}=0.$
\begin{theorem}
The explicit form of Euler-Lagrange PDEs attached to the Lagrangian
${\mathcal L}(\nabla , \partial \nabla )$ are
$$(\delta _u^i\delta ^v_{[k}\delta ^b_{l]}
\Gamma ^{w }_{bj}+
\delta ^w_j\delta ^a_{[k}\delta ^v_{l]}\Gamma ^{i }_{au}
)
R^p_{qrs}g^{jq}g^{kr}g^{ls}g_{ip}\sqrt{det (g_{ij})}
$$
$$-D_{x^t}\left [
\delta ^{t}_{[k}\delta ^v_{l]}
 R^p_{qrs}
g^{wq}g^{kr}g^{ls}g_{up}
\sqrt{det (g_{ij})}\right]=0.$$
\end{theorem}

On the Riemannian manifold $(M,g=(g_{ij}))$, we introduce
the square of the norm $L=\|Riem^g\|^2= g^{ip}g^{jq}g^{kr}g^{ls}R_{ijkl}R_{pqrs}$
which is of second order with respect to $g_{ij}$ and of order zero with respect
to $g^{ij}$. In this way the Riemann - flatness deviation is either the functional
$I(g)=\int \|Riem^g\|^2d\mu$ or the functional $I(g^{-1})=\int \|Riem^g\|^2d\mu$.
For $I(g)$ the extremals are solutions of Euler-Lagrange PDEs
$$\frac{\partial \mathcal{L}}{\partial g_{mn}}- D_{x^l} \,\,\frac{\partial \mathcal{L}}{\partial(\partial_{x^l}g_{mn})}
+ D_{x^k}D_{x^l} \,\,\frac{\partial \mathcal{L}}{\partial(\partial_{x^k}\partial_{x^l}g_{mn})}=0,$$
while for $I(g^{-1})$ the Euler-Lagrange PDEs are reduced to
$\frac{\partial \mathcal{L}}{\partial g^{mn}}=0.$

\begin{theorem}
The extremals $g=(g_{ij})$ of the Lagrangian ${\cal L}(g^{-1})$
are solutions of the PDEs system
$$2\delta ^c_{[k}\delta ^d_{l]}
R_{pqrs}
g^{ap}g^{bq}g^{kr}g^{ls} g_{nv }g_{wm}\Gamma _{bc}^{v }\Gamma _{ad}^{w}
$$
$$
+2R_{ijkl}R_{pqrs}(\delta _{m}^i\delta _{n}^{p}g^{jq}+
\delta _{m}^j\delta _{n}^{q}g^{ip})g^{kr}g^{ls}
-\frac {1}{2}R_{ijkl}R_{pqrs}g^{ip}g^{jq}g^{kr}g^{ls}\,\,g_{mn}=0.
$$
\end{theorem}

\begin{theorem}
The extremals $g=(g_{ij})$ of the Lagrangian
$${\mathcal L}(g, \partial g,  \partial ^2g)= ||Riem {}^{g}||^2\sqrt{det (g_{ij})} $$
are solutions of the PDEs system
$$
R_{pqrs}g^{kr}g^{ls} \Big[2(-\Gamma ^{n}_{jk}\Gamma ^m_{il}-\Gamma ^{n}_{jl}\Gamma ^m_{ik})
g^{ip}g^{jq}
$$
$$-R_{ijkl}\Big(2(g^{mi}g^{np}g^{jq}+
g^{ip}g^{mj}g^{nq})
-\frac{1}{2}g^{nm}g^{ip}g^{jq}\Big)\Big]\sqrt{det (g_{ij})}$$
$$-D_{x^h}\Big[
\delta ^a_{i} \delta ^b_{j}
\delta ^c_{[k}\delta ^d_{l]} \Big[\Gamma ^m_{ad}(\delta ^h_c
\delta ^n_b+\delta ^h_b\delta ^n_c)-
\Gamma ^h_{ad}\delta ^m_b\delta ^n_c$$
$$+
\Gamma ^m_{bc}(\delta ^h_a
\delta ^n_d+\delta ^h_d\delta ^n_a)-\Gamma ^h_{bc}
\delta ^m_a\delta ^n_d
\Big]
 R_{pqrs}g^{ip}g^{jq}g^{kr}g^{ls}\sqrt{det (g_{ij})}\Big]$$
$$+D^2_{x^hx^{a}}\left[g^{ip}g^{jq}g^{kr}g^{ls}\delta ^m_{[i}\delta ^h_{j]}
\delta ^n_{[k}\delta ^{a }_{l]}R_{pqrs}\sqrt{det (g_{ij})}
\right]=0.
$$
\end{theorem}

\subsection{Ricci - flatness deviation}
Let $\nabla$ be an equiaffine connection of components $\Gamma^i_{jk}$ and $g=(g_{ij})$
be a Riemannian metric.
On the smooth oriented manifold $(M,\nabla,g)$,
let us consider the Lagrangian density $L= \|Ric^\nabla\|^2= g^{ik} g^{jl} R_{ij}R_{kl}$
(square of the norm, first order in $\Gamma^i_{jk}$)
and the functional (Ricci - flatness deviation)
$I(\nabla)=\int_M \|Ric^\nabla\|^2d\mu.$
The Euler-Lagrange PDEs are
$\frac{\partial \mathcal{L} }{\partial \Gamma^l_{mn}}- D_{x^r}\frac{\partial \mathcal{L} }{\partial(\partial_{x^r} \Gamma^l_{mn})}=0.$

\begin{theorem} Let
$R_{ij} =\mathcal{P}^{ps}_{qj}\left(\frac{\partial}{\partial x^p}\Gamma^q_{is}+
\Gamma^q_{pn} \Gamma^n_{is}\right).$
The extremals $\Gamma^i_{jk}$ of the Lagrangian
$\mathcal{L}(\nabla , \partial \nabla )= g^{ik} g^{jl} R_{ij}R_{kl}\,\, \sqrt{\det(g_{ij})},$
are solutions of PDEs system
$$[\delta ^{v}_u
\Gamma ^{w }_{ij}-
\delta ^{v}_j\Gamma ^{w }_{iu}
+\delta ^v_i(\delta ^w_j\Gamma ^c_{cu}-\Gamma ^w_{ju})]
R_{kl}g^{ik}g^{jl}\sqrt{det (g_{ij})}
$$
$$-D_{x^t}\left(
\delta ^{t}_{[u}\delta ^w_{j]}
R_{kl}
g^{vk}g^{jl}
\sqrt{det (g_{ij})}\right)=0.$$
\end{theorem}

%%%%%%%%%%%%%%%%%%%%%%%%
%%%%%%%%%%%%%%%%%%%%%%%%%%%%%

The Ricci tensor field of a connection derived from a metric is always symmetric.
On the Riemannian manifold $(M,g=(g_{ij}))$,
let us consider the Lagrangian density $L= \|Ric^g\|^2= g^{ik} g^{jl} R_{ij}R_{kl}$
(square of the norm)  which is of second order in $g_{ij}$ and order zero in $g^{ij}$.
The Ricci - flatness deviation is described either by the functional
$I(g)=\int_M \|Ric^g\|^2d\mu$ or by the functional $I(g^{-1})=\int_M \|Ric^g\|^2d\mu$.

For $I(g)$ the extremals are solutions of Euler-Lagrange PDEs
$\frac{\partial \mathcal{L}}{\partial g_{mn}}- D_{x^l} \,\,\frac{\partial \mathcal{L}}{\partial(\partial_{x^l}g_{mn})}
+ D_{x^k}D_{x^l} \,\,\frac{\partial \mathcal{L}}{\partial(\partial_{x^k}\partial_{x^l}g_{mn})}=0.$
To simplify, we work first with $I(g^{-1})$ since the Euler-Lagrange PDEs determined by
$\mathcal{L}= g^{ik} g^{jl} R_{ij}R_{kl}\,\, \sqrt{\det(g_{ij})}$ are reduced to
$\frac{\partial \mathcal{L}}{\partial g^{mn}}=0.$

\begin{theorem}
We fix an harmonic coordinate system.
The extremals $g=(g_{ij})$ of the functional $I(g^{-1})$
are solutions of PDEs system
$$2g^{ik}R_{im}R_{kn}+2g^{ik}g^{jl}R_{kl}
\Big(-\frac{1}{2}\frac{\partial^2g_{ij}}{\partial x^m\partial x^n}+g_{cd}\Gamma _{im}^c\Gamma _{nj}^d +g^{ab}
g_{md}g_{nc}\Gamma _{ia}^{c }\Gamma _{bj}^d\Big)$$
$$-\frac{1}{2}g^{ik}g^{jl}R_{ij}R_{kl}\,\,g_{mn}=0.$$
\end{theorem}

\begin{theorem}
We fix an harmonic coordinate system. The extremals $g=(g_{ij})$ of the
functional $I(g)$ are solutions of PDEs system
$$\sqrt{det (g_{ij})}\Big[R_{ij}R_{kl}(-2g^{mi}g^{nk}g^{jl}+\frac {1}{2}g^{ik}g^{jl}g^{mn})+2g^{ik}g^{jl}$$
$$\times R_{ij}\Big[-g^{mp}g^{nq}\left(-\frac {1}{2}
\frac {\partial ^2g_{kl}}{\partial x^p\partial x^q}
+g_{rs}\Gamma _{kp}^r\Gamma _{ql}^s\right)
-g^{pq}\Gamma _{kp}^{n}\Gamma _{ql}^{m}
\Big]\Big]$$
$$-D_{x^h}\Big[\sqrt{det (g_{ij})}R_{kl}g^{ab}g^{ik}g^{jl}
\Big[ \Big(\delta ^n_d(\delta ^{h}_{a}\delta ^m_{i}+
 \delta ^{h}_i\delta ^m_{a})-
 \delta ^{h}_d\delta ^m_i\delta ^n_{a}\Big)
\Gamma _{bj}^d$$
$$+\Big(\delta ^n_d\Big(\delta ^{h}_{j}\delta ^m_{b}+
 \delta ^{h}_b\delta ^m_{j})-
 \delta ^{h}_d\delta ^m_b\delta ^n_{j}\Big)
\Gamma _{ia}^d
\Big]\Big]
+D^2_{x^hx^{t }}\Big[\sqrt{det (g_{ij})}
g^{ht }g^{mk}g^{nl}R_{kl}\Big]=0.$$
\end{theorem}

\subsection{Scalar curvature - flatness deviation}
Let $\nabla$ be an equiaffine connection of components $\Gamma^i_{jk}$ and $g=(g_{ij})$
be a Riemannian metric. On the manifold $(M,\nabla,g)$, we introduce the functional (total scalar curvature)
$I(\nabla)= \int_M \mathcal{R}^\nabla\,\,d\mu,$
where $\mathcal{R}^\nabla = g^{ij}R_{ij}$, and the Lagrangian $\mathcal{L}= \mathcal{R}^\nabla\,\sqrt{\det (g_{ij})}$
is of first order with respect to $\Gamma^i_{jk}$. The general Euler-Lagrange PDEs are
$\frac{\partial \mathcal{L} }{\partial \Gamma^l_{mn}}- D_{x^r}\frac{\partial \mathcal{L} }{\partial(\partial_{x^r} \Gamma^l_{mn})}=0.$
\begin{theorem}
The Euler–Lagrange PDEs attached to the functional $I(\nabla)$, i.e.,
 to the Lagrangian $\mathcal{L}= g^{ik} R_{ik}\,\,\sqrt{\det (g_{ik})}$, are
$$\mathcal{P}^{ps}_{qk} \left[   g^{ik}\left(\delta^q_l \delta^m_p \Gamma^n_{is} + \delta^m_i \delta^n_s \Gamma^q_{pl}\right)\sqrt{\det (g_{ab})} - \delta^q_l\delta^n_s D_{x^p}\left(g^{mk}\sqrt{\det (g_{ab})}\right)\right]=0.$$
\end{theorem}
\begin{proof} Since $\mathcal{L}= g^{ik} R_{ik}\,\,\sqrt{\det (g_{ik})}$,
$R_{ik}=\mathcal{P}^{ps}_{qk}\left(\frac{\partial}{\partial x^p}\Gamma^q_{is}+
\Gamma^q_{pr} \Gamma^r_{is}\right)$, and
$$
\frac{\partial \Gamma^i_{jk}}{\partial \Gamma^l_{mn}}=\delta^i_l\delta^m_j\delta^n_k,\,\,\,\frac{\partial \left(\partial_{x^p}\Gamma^q_{is}\right)}{\partial \left(\partial_{x^r}\Gamma^l_{mn}\right)}=\delta^r_p\delta^q_l\delta^m_i\delta^n_s,
$$
$$\frac{\partial R_{ik}}{\partial \Gamma^l_{mn}}=\mathcal{P}^{ps}_{qk}\left(\delta^q_l \delta^m_p \Gamma^n_{is} + \delta^m_i \delta^n_s \Gamma^q_{pl}\right);\,
\frac{\partial R_{ik}}{\partial_{x^r} \Gamma^l_{mn}}=\mathcal{P}^{ps}_{qk}\,\,\,\delta^r_p\delta^q_l\delta^m_i\delta^n_s$$
we obtain the PDEs in the Theorem.
\end{proof}

On a smooth oriented Riemannian manifold $(M,g=(g_{ij}))$,
we attach the functional (total scalar curvature)
$I(g)=\int_M R^g\,\,d\mu,$ $R^g=g^{ij}R_{ij}$.
Here the Lagrangian  $\mathcal{L}= g^{ij} R_{ij}\,\sqrt{\det (g_{ij})}$ is of the second order
with respect to $g_{ij}$, and of order zero with respect to $g^{ij}$.
In dimension two, this is a topological quantity, namely the Euler
characteristic of the Riemann surface by the Gauss-Bonnet formula.
In $n\geq 3$ dimension we prefer the functional
$I(g^{-1})=\int_M R^g\,\,d\mu.$
\begin{theorem}
 The Euler–Lagrange PDEs attached to the functional $I(g^{-1})$, $n\geq 3$,  i.e.,
 to the Lagrangian $\mathcal{L}= g^{ij} R_{ij}\,\,\sqrt{\det (g_{ij})}$, are
 $R_{ij}=0.$
\end{theorem}

\begin{proof} The Euler-Lagrange PDEs are $\frac{\partial \mathcal{L}}{\partial g^{mn}}=0$,
where $\mathcal{L}= g^{ij} R_{ij}\,\,\sqrt{\det (g_{ij})}$.
On the other hand, we have
$$\frac{\partial \det (g_{ij)}}{\partial g^{mn}}= - \det (g_{ij})g_{mn},\,\,\,
\frac{\partial g_{jl}}{\partial g^{mn}}=-g_{mj}g_{nl},\,\,\,\frac{\partial g^{jl}}{\partial g^{mn}}=\delta^j_m\delta^l_n.$$
The term $g^{ij} \frac{\partial R_{ij}}{\partial g^{mn}}\,\,\sqrt{\det (g_{ij})}$ is of divergence type,
and it has no contribution to the Euler-Lagrange equations. Consequently
$\frac{\partial R}{\partial g^{mn}}= \frac{\partial (g^{ij} R_{ij})}{\partial g^{mn}}=R_{mn}.$
Finally, we obtain the explicit Euler--Lagrange PDEs as $R_{ij}=0$.
\end{proof}

\begin{theorem} \cite{DK}
The solutions of the problem "$\min_{g_{ij}}\int_M R^g\,\,d\mu$ subject to
$\int_M d\mu=1, n\geq 3$", are solutions of Einstein $nD$ PDEs $\displaystyle R_{ij}= \frac{R}{n}\,\,g_{ij}$.
\end{theorem}
\begin{proof} We use the Lagrangian $\mathcal{L}= g^{ij} R_{ij}\,\,\sqrt{\det (g_{ij})}- \lambda \sqrt{\det (g_{ij})}$,
where $\lambda$ is a constant multiplier. Taking the variations with respect to $g^{mn}$, we
obtain
$$R_{ij}-\frac{R-\lambda}{2}\,\,\,g_{ij}=0.$$
The hypothesis $n\geq 3$ and $\lambda = c$ implies that $R$ is constant.
We replace $R$, respectively $R_{ij}$, in $\int_M R^g\,\,d\mu$ and we obtain
$R\,\,vol(M)=\int_M R^g\,\,d\mu= \frac{R-\lambda}{2}\,\, n \,\,vol(M).$
Consequently, $\lambda =\frac{(n-2)R}{n}$ and $R_{ij}= \frac{R}{n}\,\,g_{ij}$.
\end{proof}

The exact solutions of Einstein PDEs were discussed in \cite{M}, \cite{SKM}.
In dimension four, there are topological obstructions to the existence of Einstein metrics.

On a smooth oriented Riemannian manifold $(M,g=(g_{ij}))$,
we attach a scalar curvature - flatness deviation either by the functional
$I(g)=\int_M \left(R^g\right)^2d\mu$ or as the functional $I(g^{-1})=\int_M \left(R^g\right)^2d\mu$.

\begin{theorem}
 The Euler–Lagrange PDEs attached to functional $I(g^{-1})$,
 i.e. to the Lagrangian $\mathcal{L}= (g^{ij}R_{ij})^2 \,\,\sqrt{\det(g_{ij})}$ (zero order with respect to $g^{ij}$),
 are either $R=0$ or $R_{ij}=0$.
\end{theorem}

\begin{corollary}
The solutions of PDEs $R=0$ or $R_{ij}=0$
are Euler-Lagrange prolongations of metrics $g_{ij}(x)=c_{ij}$.
\end{corollary}

\section{Conclusions}

In this paper were studied operators with partial derivatives corresponding to
the most important geometric objects in Riemannian geometry and affine differential geometry.
The index form technique facilitates the understanding of the significance of the geometric PDEs 
and of the Lagrangian densities attached to them using the Riemannian metrics.

For all the above geometric PDEs, only existence and uniqueness properties have been discussed.
Particularly, we have referred to non-tensorial PDEs of Riccati type, whose unknowns
are the components of the affine connection.
In the Riemannian case, the Lagrange-type densities are the squares of the norms of
important geometric objects: connection, curvature tensor field, Ricci tensor field,
scalar curvature field. Also, to obtain the Euler-Lagrange PDEs,
the variations with respect to the metric $g$ or its inverse $g^{-1}$ are used alternatively.

In light of the above discussion, if one is able to say something about the solution of a PDEs system
whose solution is a Riemannian metric or an affine connection, one could perhaps say
something interesting about the behaviour of the manifold and its structure.

We are convinced that the present paper will trigger new research in differential geometry.

{\bf Acknowledgements}:

{\bf Funding}:

{\bf Conflicts of Interest}: The authors declare no conflict of interest.

{\bf Scientific Guarantee}: University Politehnica of Bucharest, Academy of Romanian Scientists.

\noindent {\it Author's addresses}:

{Iulia Elena Hirica Associate Professor, Gabriel Pripoae Professor,
University of Bucharest, Faculty of Mathematics and Computer Science
14 Academiei Str., RO-010014,  Bucharest 1, Romania} \\
{https://orcid.org/0000-0002-1339-995X}; {https://orcid.org/0000 - 0002 - 3315 - 3762}\\
E-mails:  {\tt ihirica@fmi.unibuc.ro} , {\tt gabrielpripoae@yahoo.com}

\medskip

Constantin Udriste Professor, Ionel Tevy Professor, Department of Mathematics and Informatics, Faculty of Applied
Sciences, University Politehnica of Bucharest, Splaiul Independentei 313, RO-060042, Bucharest 6, Romania
\\{https://orcid.org/0000-0002-7132-6900}; {https://orcid.org/0000 - 0001 - 9160 - 9185}
\\E-mails: {\tt udriste@mathem.pub.ro} , {\tt anet.udri@yahoo.com} , \\{\tt vascatevy@yahoo.fr}
\medskip

\end{document}